\documentclass[a4paper,12pt,reqno,twoside]{amsart}
\usepackage{Kloosterman_depth}
\usepackage[colorlinks,linkcolor=blue,anchorcolor=blue,citecolor=blue,backref=page]{hyperref}
\usepackage{float}
\floatstyle{boxed}
\newfloat{Fig}{htbp}{Fig}[section]
\floatname{Fig}{Fig}
\usepackage[norefs,nocites]{refcheck}

\def\cI{{\mathcal I}}
\def\cJ{{\mathcal J}}

\begin{document}%
\title[Kloosterman paths of prime powers moduli II]{Kloosterman paths of prime powers moduli,  II}
\author[G. Ricotta]{Guillaume Ricotta}
\address{Universit\'{e} de Bordeaux \\
Institut de Math\'{e}matiques de Bordeaux \\
351, cours de la Lib\'{e}ration \\
33405 Talence cedex \\
France}
\email{Guillaume.Ricotta@math.u-bordeaux.fr}
\author[E. Royer]{Emmanuel Royer}
\address{Universit\'e Clermont Auvergne \\
CNRS \\
LMBP \\
Clermont-Ferrand \\
France}
\email{emmanuel.royer@math.cnrs.fr}
\author[I. Shparlinski]{Igor Shparlinski}
\address{The University of New South Wales \\
Sydney NSW 2052 \\
Australia}
\email{igor.shparlinski@unsw.edu.au}
\date{Version of \today} 
\subjclass{11T23, 11L05, 60F17, 60G17, 60G50.}
\keywords{Kloosterman sums, moments, probability in Banach spaces.}
\begin{abstract}
G. Ricotta and E. Royer (2018) have recently proved that the polygonal paths joining the partial sums of the normalized classical Kloosterman sums $S\left(a,b;p^n\right)/p^{n/2}$ converge in law in the Banach space of complex-valued continuous function on $[0,1]$ to an explicit random Fourier series as $(a,b)$ varies over $\left(\mathbb{Z}/p^n\mathbb{Z}\right)^\times\times\left(\mathbb{Z}/p^n\mathbb{Z}\right)^\times$, $p$ tends to infinity among the odd prime numbers and $n\geq 2$ is a fixed integer. This is the analogue of the result obtained by E.~Kowalski and W.~Sawin (2016) in the prime moduli case. The purpose of this work is to prove a convergence law in this Banach space as only $a$ varies over $\left(\mathbb{Z}/p^n\mathbb{Z}\right)^\times$, $p$ tends to infinity among the odd prime numbers and $n\geq 31$ is a fixed integer.
\end{abstract}
%
\maketitle
%
\begin{center}
\textit{In memory of Alexey Zynkin.}
\end{center}
%
\section{Introduction and statement of the results}%
%
\subsection{Background}
Let $p$ be an odd prime number and $n\geq 1$ be an integer. For $a$ and $b$ in $\mathbb{Z}/p^n\mathbb{Z}$, the corresponding normalized Kloosterman sum of modulus $p^n$ is the real number given by
\[
\mathsf{Kl}_{p^n}(a,b)\coloneqq\frac{1}{p^{n/2}}S\left(a,b;p^n\right)=\frac{1}{p^{n/2}}\sum_{\substack{1\leq x\leq p^n \\
p\nmid x}}\
e\left(\frac{ax+b\overline{x}}{p^n}\right), 
\]
where as usual $\overline{x}$ stands for the inverse of $x$ modulo $p^n$ and we  also define $e(z)\coloneqq\exp{(2i\pi z)}$ for any complex number $z$. Recall that its absolute value is less than $2$ by its explicite formula (see \cite[Lemma 4.6]{MR3854900} for instance). For $a$ and $b$ in $\left(\mathbb{Z}/p^n\mathbb{Z}\right)^\times$, the associated partial sums are the $\varphi(p^n)=p^{n-1}(p-1)$ complex numbers
\[
\mathsf{Kl}_{j;p^n}(a,b)\coloneqq\frac{1}{p^{n/2}}\sum_{\substack{1\leq x\leq j \\
p\nmid x}}e\left(\frac{ax+b\overline{x}}{p^n}\right)
\]
for $j$ in $\cJ_p^n\coloneqq\left\{j\in\{1,\dots,p^n\}, p\nmid j\right\}$. If we write \(\cJ_p^n=
\{j_1,\dotsc,j_{\varphi(p^n)}\}\), the corresponding Kloosterman path $\gamma_{p^n}(a,b)$ is defined by %
\[%
\gamma_{p^n}(a,b)=\bigcup_{i=1}^{\varphi(p^n)-1}\left[\mathsf{Kl}_{j_i;p^n}(a,b),\mathsf{Kl}_{j_{i+1};p^n}(a,b)\right]. %
\]
This is the polygonal path obtained by concatenating the closed segments
\[
\left[\mathsf{Kl}_{j_1;p^n}(a,b),\mathsf{Kl}_{j_2;p^n}(a,b)\right]
\]
for $j_1$ and $j_2$ two consecutive indices in $\cJ_{p^n}$. Finally, one defines a continuous map on the interval $[0,1]$
\[
t\mapsto\mathsf{Kl}_{p^n}(t;(a,b))
\]
by parametrizing the path $\gamma_{p^n}(a,b)$, each segment 
\[
\left[\mathsf{Kl}_{j_1;p^n}(a,b),\, \mathsf{Kl}_{j_2;p^n}(a,b)\right]
\] for $j_1$ and $j_2$ two consecutive indices in $\cJ_{p^n}$ being parametrized linearly by an interval of length $1/(\varphi(p^n)-1)$.

The function $(a,b)\mapsto\mathsf{Kl}_{p^n}(\ast;(a,b))$ is viewed as a random variable on the probability space $\left(\mathbb{Z}/p^n\mathbb{Z}\right)^\times\times\left(\mathbb{Z}/p^n\mathbb{Z}\right)^\times$ endowed with the uniform probability measure with values in the Banach space of complex-valued continuous functions on $[0,1]$ endowed with the supremum norm, say $C^0([0,1],\C)$.

Let $\mu$ be the probability measure given by
\[
\mu=\frac{1}{2}\delta_0+\mu_0
\]
for the Dirac measure $\delta_0$ at $0$ and
\[
\mu_0(f)=\frac{1}{2\pi}\int_{x=-2}^2\frac{f(x)\mathrm{d}x}{\sqrt{4-x^2}}
\]
for any real-valued continuous function $f$ on $[-2,2]$.

Let $\left(U_h\right)_{h\in\Z}$ be a sequence of independent identically distributed random variables of probability law $\mu$ and let $\mathsf{Kl}$ be the $C^0([0,1],\C)$-valued random variable defined by
\[
\forall t\in[0,1],\quad \mathsf{Kl}(t)(\ast)=tU_0(\ast)+\sum_{h\in\Z^\ast}\frac{e(ht)-1}{2i\pi h}U_h(\ast)
\]
where the series should be summed with symmetric partial sums. The basic properties of this random Fourier series are given in~\cite[Proposition~3.1]{MR3854900}.

In~\cite{MR3854900}, it has been 
 proved that the sequence of $C^0([0,1],\C)$-valued random variables $\mathsf{Kl}_{p^n}(\ast;(\ast,\ast))$ on $\left(\Z/p^n\Z\right)^\times\times\left(\Z/p^n\Z\right)^\times$ converges in law\footnote{See~\cite[Appendix~A]{MR3854900} for a precise definition of the convergence in law in the Banach space $C^0([0,1],\C)$.} to the $C^0([0,1],\C)$-valued random variable $\mathsf{Kl}$ as $p$ tends to infinity among the prime numbers and $n\geq 2$ is a fixed integer. This is the analogue of the result proved by E.~Kowalski and W.~Sawin in~\cite{KoSa} when $n=1$ where a different random Fourier series occur, the measure $\mu$ being replaced by the Sato-Tate measure.

\subsection{Our results and approach}
The purpose of this work is the following substantial refinement of~\cite{MR3854900}.

\begin{theoint}[Convergence in law]\label{theo_A}
Let $b_0$ be a fixed non-zero integer, $n\geq 31$ be a fixed integer and $p$ be an odd prime number. The sequence of $C^0([0,1],\C)$-valued random variables $\mathsf{Kl}_{p^n}(\ast;(\ast,b_0))$ on $\left(\Z/p^n\Z\right)^\times$ converges in law to the $C^0([0,1],\C)$-valued random variable $\mathsf{Kl}$ as $p$ tends to infinity among the odd prime numbers.
\end{theoint}
This result is not a purely technical problem. This question should be seen as a very challenging one already highlighted as a key open problem in \cite{KoSa}. Note that the case of prime moduli remains open.

The general strategy to prove Theorem~\ref{theo_A} is the one used in~\cite{KoSa} and in~\cite{MR3854900}. It consists of two distinct steps.

First of all, the convergence in the sense of finite distributions\footnote{See~\cite[Appendix~A]{MR3854900} for a precise definition of the convergence in the sense of finite distributions in the Banach space $C^0([0,1],\C)$.} of the sequence of $C^0([0,1],\C)$-valued random variables $\mathsf{Kl}_{p^n}(\ast;(\ast,b_0))$ on $\left(\Z/p^n\Z\right)^\times$ to the $C^0([0,1],\C)$-valued random variable $\mathsf{Kl}$ as $p$ tends to infinity among the odd prime numbers is proved. This result is nothing else than~\cite[Theorem A]{MR3854900} and heavily relies on A.~Weil's version of the Riemann hypothesis in one variable, see~\cite{Lor}. Note that $n\geq 2$ is fixed in this work, although most of our ingredients work 
for arbitrary $n$ as well. Indeed, the only place where $n$ has to be a fixed integer is when using \cite[Theorem A]{MR3854900} page \pageref{sec_proofs}.

Then, to deduce the convergence in law from the convergence of finite distributions, one has to prove that the sequence of $C^0([0,1],\C)$-valued random variables $\mathsf{Kl}_{p^n}(\ast;(\ast,b_0))$ on $\left(\Z/p^n\Z\right)^\times$ is tight, a weak-compactness property which takes into account that $C^0([0,1],\C)$ is infinite-dimensional. See~\cite[Appendix A]{MR3854900} for a precise definition. This is the main input in this work.
\begin{theoint}[Tightness property]\label{theo_B}
Let $b_0$ be a fixed non-zero integer, $n\geq 31$ be a fixed integer and $p$ be an odd prime number. The sequence of $C^0([0,1],\C)$-valued random variables $\mathsf{Kl}_{p^n}(\ast;(\ast,b_0))$ on $\left(\Z/p^n\Z\right)^\times$ is tight as $p$ tends to infinity among the odd prime numbers.
\end{theoint}
The proof of this tightness property in Theorem~\ref{theo_B} follows the strategy outlined in~\cite{KoSa}. The core of the proof is a uniform estimate of the shape
\[
\frac{1}{p^{n/2}}\sideset{}{^\times}\sum_{x\in \cI}e\left(\frac{ax+b\overline{x}}{p^n}\right)\ll p^{-\delta}
\]
for some $\delta>0$ when $\cI$ is an interval of $\Z$ of length close to $p^{n/2}$ and $a$ and $b$ are some integers coprime with $p$. See~\cite[Remark~3.3]{KoSa} and~\cite[Page~52]{Ko} for a discussion on such issues in the prime moduli case. It turns out that for $n$ large enough, such estimate follows from the work contained in~\cite{MR02681422}.

Finally, one can mention that it seems quite natural to consider the same questions in the regime
$p$ a fixed prime number and $n\geq 2$ tends to infinity, or even in a more complicated regime 
when both $p$ and $n$ vary. This problem, both theoretically and numerically, seems to be of completely different nature.

\subsection{Organization of the paper}
The explicit description of the Kloosterman paths and some required results proved in~\cite{MR3854900} are recalled in Section~\ref{sec_back}. Section~\ref{sec_koro} deals with Korolev's estimate for short Kloosterman sums of powerful moduli. The tightness condition is proved in Section~\ref{sec_tight}. 
Theorems~\ref{theo_A} and~\ref{theo_B} are proved in Section~\ref{sec_proofs}.

\subsection{Notation}
The main parameter in this paper is an odd prime $p$, which tends to infinity. Thus, if $f$ and $g$ are some $\C$-valued function of one real variable then the notations 
\[
f(p)=O_A(g(p))\quad \text{or} \quad  f(p)\ll_A g(p) \quad \text{or} \quad g(p) \gg_A f(p)
\] 
mean that $\abs{f(p)}$ is smaller than a constant, which only depends on $A$, times $g(p)$ at least for $p$ large enough. 
\par
Throughtout the paper, $n\geq 2$ is a fixed integer and $b_0$ is a fixed non-zero integer.
\par
For any real number $x$ and integer $k$, 
$$
e_k(x)\coloneqq\exp{\left(\frac{2i\pi x}{k}\right)}.
$$
\par
For any finite set $S$, $\vert S\vert$ stands for its cardinality. 
\par
We will denote by $\varepsilon$ an absolute positive constant whose definition may change from one line to the next one.
\par
The notation $\sideset{}{^\times}\sum$ means that the summation is over a set of integers coprime with $p$.
%
\section{Background on the Kloosterman path}\label{sec_back}%
\subsection{Explicit description of the Kloosterman path}%
Let us recall the construction of the Kloosterman path $\gamma_{p^n}(a,b_0)$ given in~\cite[Section~2]{MR3854900} for $a$ in $\left(\mathbb{Z}/p^n\mathbb{Z}\right)^\times$. %

We enumerate the partial Kloosterman sums and define \(z_j((a,b_0);p^n)\) to be the \(j\)-th term of \(\left(\mathsf{Kl}_{j;p^n}(a,b_0)\right)_{j\in \cJ_p^n}\). More explicitly, we organise the partial Kloosterman sums in \(p^{n-1}\) blocks each of them containing \(p-1\) successive sums. For $1\leq k\leq p^{n-1}$, the \(k\)-th block contains 
\[\mathsf{Kl}_{(k-1)p+1;p^n}(a,b_0),\dotsc,\mathsf{Kl}_{kp-1;p^n}(a,b_0).\]
 These sums are numbered by defining %
\[%
z_{(k-1)(p-1)+\ell}((a,b_0);p^n)=\mathsf{Kl}_{(k-1)p+\ell;p^n}(a,b_0)\qquad (1\leq \ell\leq p-1). %
\]
It implies that the enumeration is given by %
$$
z_{j}((a,b_0);p^n)=\mathsf{Kl}_{j+\left\lfloor\frac{j-1}{p-1}\right\rfloor;p^n}(a,b_0)\qquad (1\leq j<\varphi(p^n)). 
$$

For any \(j\in\{1,\dotsc,\varphi(p^n)-1\}\), we parametrize the open segment \[\left(z_{j}((a,b_0);p^n),z_{j+1}((a,b_0);p^n)\right]\] and obtain the parametrization of \(\gamma_{p^n}(a,b_0)\) given by %
\begin{equation}\label{eq_21}
\begin{split}
\forall t\in[0,1],\quad & \mathsf{Kl}_{p^n}(t;(a,b_0))\\
& =\alpha_j((a,b_0);p^n)\left(t-\frac{j-1}{\varphi(p^n)-1}\right)+z_{j}((a,b_0);p^n) %
\end{split}
\end{equation}
with %
\[%
\alpha_j((a,b_0);p^n)=(\varphi(p^n)-1)\left(z_{j+1}((a,b_0);p^n)-z_{j}((a,b_0);p^n)\right) %
\]
and 
\[%
j=\left\lceil\left(\varphi(p^n)-1\right)t\right\rceil\quad\text{namely}\quad \frac{j-1}{\varphi\left(p^n\right)-1}<t\leq\frac{j}{\varphi\left(p^n\right)-1}.  %
\]
Since \(\left(z_{j}((a,b_0);p^n),z_{j+1}((a,b_0);p^n)\right]\) has length \(p^{-n/2}\), we have %
\begin{equation}\label{eq_tec_0}
\abs{\alpha_j((a,b_0);p^n)}\leq\frac{\varphi(p^n)-1}{p^{n/2}}.
\end{equation}
\subsection{Approximation of the Kloosterman path}%
For $a$ in $\left(\mathbb{Z}/p^n\mathbb{Z}\right)^\times$, let us define a step function on the segment $[0,1]$ by, for any $k\in\left\{1,\dots,p^{n-1}\right\}$,
\begin{equation}\label{eq_Kltilde}
\forall t\in\left(\frac{k-1}{p^{n-1}},\frac{k}{p^{n-1}}\right],\quad\widetilde{\mathsf{Kl}_{p^n}}(t;(a,b_0))\coloneqq\frac{1}{p^{n/2}}\sum_{1\leq x\leq x_k(t)}^{\mathstrut\quad\times}e_{p^n}\left(ax+b_0\overline{x}\right).
\end{equation}
where
\[
x_k(t)\coloneqq\varphi(p^n)t+k-1.
\]
In addition, let us define for $h$ in $\mathbb{Z}/p^n\mathbb{Z}$ and $1\leq k\leq p^{n-1}$,
\begin{equation}\label{eq_Fourier}
\forall t\in\left(\frac{k-1}{p^{n-1}},\frac{k}{p^{n-1}}\right],\quad\alpha_{p^n}(h;t)\coloneqq\frac{1}{p^{n/2}}\sum_{1\leq x\leq x_k(t)}e_{p^n}\left(hx\right).
\end{equation}
These coefficients are nothing else than the discrete Fourier coefficients of the finite union of intervals given by $1\leq x\leq x_k(t)$ with $(p,x)=1$ for $1\leq k\leq p^{n-1}$.

The sequence of random variables $\widetilde{\mathsf{Kl}_{p^n}}(\ast;(\ast,b_0))$ on $\left(\Z/p^n\Z\right)^\times$ is an approximation of the sequence of $C^0([0,1],\C)$-valued random variables $\mathsf{Kl}_{p^n}(\ast;(\ast,b_0))$ on $\left(\Z/p^n\Z\right)^\times$ in the sense that
\begin{equation}\label{eq_bemol}
\left\vert\mathsf{Kl}_{p^n}(t;(a,b_0))-\widetilde{\mathsf{Kl}_{p^n}}(t;(a,b_0))\right\vert\leq\frac{6}{p^{n/2}}
\end{equation}
for any $a$ in $\left(\mathbb{Z}/p^n\mathbb{Z}\right)^\times$ and any $t\in[0,1]$. 
See~\cite[Equation~(2.3)]{MR3854900}.

Finally, by~\cite[Lemma~4.2]{MR3854900} and~\cite[Remark~4.5]{MR3854900}, 
\begin{equation}\label{eq_bemol2}
\widetilde{\mathsf{Kl}_{p^n}}(t;(a,b_0))\ll\log{\left(p^n\right)}.
\end{equation}
Note that~\eqref{eq_bemol} and~\eqref{eq_bemol2} are essentially a consequence of the very classical completion method.
\section{On Korolev's estimate for short Kloosterman sums of powerful moduli}\label{sec_koro}%
A key ingredient in this work is the following particular case of an estimate proved by M.A.~Korolev for short Kloosterman sums of powerful moduli, see \cite[Theorem~1]{MR02681422}.

\begin{proposition}[Korolev's estimate~\cite{MR02681422}]\label{propo_koro}
Let $a$ $b$ and $c$ be integers and $N$ be a positive integer. If
\begin{equation}\label{eq_koro_assum}
\max\left(p^{15}, \exp{\left(\gamma_1\left(\log{\left(p^n\right)}\right)^{2/3}\right)}\right)\leq N\leq p^{n/2}
\end{equation}
then
\[
\left\vert\hspace{0.2cm}\sideset{}{^\times}\sum_{c<x\leq c+N}e_{p^n}\left(ax+b\overline{x}\right)\right\vert\leq N\exp{\left(-\gamma_2\frac{\left(\log{(N)}\right)^3}{\left(\log{\left(p^n\right)}\right)^2}\right)}
\]
where $\gamma_1=900$ and $\gamma_2=160^{-4}$.
\end{proposition}
\begin{corollary}\label{propo_koro_2}
Let $a$ $b$ and $c$ be integers and $N$ be a positive integer. If $n\geq 31$ then
\[
\left\vert\hspace{0.2cm}\sideset{}{^\times}\sum_{c<x\leq c+N}e_{p^n}\left(ax+b\overline{x}\right)\right\vert\leq 4N\exp{\left(-\gamma_2\frac{\left(\log{(N)}\right)^3}{\left(\log{\left(p^n\right)}\right)^2}\right)}.
\]
\end{corollary}
\begin{proof}[\proofname{} of Corollary~\ref{propo_koro_2}]%
By Proposition \ref{propo_koro}, one can assume that
\begin{equation*}
N<\min{\left(\exp{\left(\gamma_1\left(\log{\left(p^n\right)}\right)^{2/3}\right)},p^{n/2}\right)},
\end{equation*}
which implies that
\begin{equation*}
\exp{\left(-\gamma_2\gamma_1^3\right)}\leq\exp{\left(-\gamma_2\frac{\left(\log{(N)}\right)^3}{\left(\log{\left(p^n\right)}\right)^2}\right)}.
\end{equation*}
Trivially, one gets
\begin{align*}
\left\vert\hspace{0.2cm}\sideset{}{^\times}\sum_{c<x\leq c+N}e_{p^n}\left(ax+b\overline{x}\right)\right\vert & \leq N=\exp{\left(\gamma_2\gamma_1^3\right)}\exp{\left(-\gamma_2\gamma_1^3\right)}N \\
& \leq 4\exp{\left(-\gamma_2\gamma_1^3\right)}N \\
& \leq 4\exp{\left(-\gamma_2\frac{\left(\log{(N)}\right)^3}{\left(\log{\left(p^n\right)}\right)^2}\right)}.
\end{align*}
\end{proof}
\begin{corollary}\label{coro_koro}
Let $a$ and $b$ be some integers and
\begin{equation}\label{eq_delta}
0<\delta\leq\min{\left(\gamma_2n/16,n/2-15\right)}.
\end{equation}
If $n\geq 31$ then for any interval $\cI$ of $\Z$ whose length satisfies
\[
p^{n/2-\delta}\leq\abs{\cI}\leq p^{n/2+\delta},
\]
one has
\[
\frac{1}{p^{n/2}}\sideset{}{^\times}\sum_{x\in \cI}e_{p^n}\left(ax+b\overline{x}\right)\ll\left(\frac{1}{p^n}\right)^{\delta/n}.
\]
\end{corollary}

\begin{proof}[\proofname{} of Corollary~\ref{coro_koro}]%
Let us denote by $N$ the length of $\cI=(c,c+N]$.

If $p^{n/2-\delta}\leq N\leq p^{n/2}$ then
\begin{equation}\label{eq_small}
\left\vert\frac{1}{p^{n/2}}\sideset{}{^\times}\sum_{x\in \cI}e_{p^n}\left(ax+b\overline{x}\right)\right\vert\leq 4\frac{N}{p^{n/2}}\exp{\left(-\gamma_2\frac{\left(\log{(N)}\right)^3}{\left(\log{\left(p^n\right)}\right)^2}\right)}\leq 4\left(\frac{1}{p^n}\right)^{\gamma_2\left(\frac{n/2-\delta}{n}\right)^3}
\end{equation}
by Corollary~\ref{propo_koro_2}. Note that~\eqref{eq_koro_assum} is satisfied since $15\leq n/2-\delta$.

Let us assume from now on that $N>p^{n/2}$ and let us denote by $k$ the ceiling part of $N/p^{n/2}$. One can decompose the interval $\cI$ into a disjoint union of the $k-1$ intervals
\[
\cI_\ell\coloneqq\left(c+(\ell-1)p^{n/2},c+\ell p^{n/2}\right], \quad 1\leq\ell\leq k-1
\] 
of lengths $p^{n/2}$ and of the interval
\[
\cI_k\coloneqq\left(c+(k-1)p^{n/2},c+N\right]
\]
of length $0<N-(k-1)p^{n/2}\leq p^{n/2}$. For any $1\leq\ell\leq k-1$,
\begin{equation}\label{eq_estiml}
\left\vert\frac{1}{p^{n/2}}\sideset{}{^\times}\sum_{x\in \cI_\ell}e_{p^n}\left(ax+b\overline{x}\right)\right\vert\leq\left(\frac{1}{p^n}\right)^{\gamma_2/8}
\end{equation}
by Proposition~\ref{propo_koro}. Note that~\eqref{eq_koro_assum} is satisfied since $15\leq n/2$.

Let us deal with the last interval $\cI_k$. If $N-(k-1)p^{n/2}<p^{n/2-\delta}$ then a trivial estimate leads to
\begin{equation}\label{eq_estim_k1}
\left\vert\frac{1}{p^{n/2}}\sideset{}{^\times}\sum_{x\in \cI_k} e_{p^n}\left(ax+b\overline{x}\right)\right\vert\leq\left(\frac{1}{p^n}\right)^{\delta/n}
\end{equation}
whereas if $p^{n/2-\delta}\leq N-(k-1)p^{n/2}\leq p^{n/2}$ then
\begin{equation}\label{eq_estim_k2}
\left\vert\frac{1}{p^{n/2}}\sideset{}{^\times}\sum_{x\in \cI_k} e_{p^n}\left(ax+b\overline{x}\right)\right\vert\leq 4\left(\frac{1}{p^n}\right)^{\gamma_2\left(\frac{n/2-\delta}{n}\right)^3}.
\end{equation}
Altogether, if $N>p^{n/2}$ then
\begin{equation}\label{eq_large}
\left\vert\frac{1}{p^{n/2}}\sideset{}{^\times}\sum_{x\in \cI} e_{p^n}\left(ax+\overline{x}\right)\right\vert\leq \left(\frac{1}{p^n}\right)^{\gamma_2/8-\delta/n}+4\left(\frac{1}{p^n}\right)^{\gamma_2\left(\frac{n/2-\delta}{n}\right)^3}+\left(\frac{1}{p^n}\right)^{\delta/n}.
\end{equation}
by \eqref{eq_estiml}, \eqref{eq_estim_k1} and \eqref{eq_estim_k2}.

By \eqref{eq_small} and \eqref{eq_large}, one gets
\begin{equation*}
\left\vert\frac{1}{p^{n/2}}\sideset{}{^\times}\sum_{x\in \cI} e_{p^n}\left(ax+\overline{x}\right)\right\vert\leq \left(\frac{1}{p^n}\right)^{\gamma_2/8-\delta/n}+4\left(\frac{1}{p^n}\right)^{\gamma_2\left(\frac{n/2-\delta}{n}\right)^3}+\left(\frac{1}{p^n}\right)^{\delta/n}\ll\left(\frac{1}{p^n}\right)^{\delta/n}
\end{equation*}
since a simple computation ensures that
\[
\frac{\delta}{n}\leq\frac{\gamma_2}{8}-\frac{\delta}{n}\leq\gamma_2\left(\frac{n/2-\delta}{n}\right)^3
\]
by \eqref{eq_delta}.
\end{proof}
\section{On the tightness property via Kolmogorov's criterion}\label{sec_tight}%
The goal of this section is to prove the following proposition.
\begin{proposition}[Tightness property]\label{propo_tightne2}
Let $n\geq 31$ be a fixed integer and $b_0$ be a fixed non-zero integer. There exists $\alpha>0$ depending only on $n$ and $\beta>0$ depending only on $\alpha$ and $n$ such that
\[
\frac{1}{\varphi\left(p^n\right)}\sum_{a\in\left(\Z/p^n\Z\right)^\times}\left\vert\mathsf{Kl}_{p^n}(t;(a,b_0))-\mathsf{Kl}_{p^n}(s;(a,b_0))\right\vert^\alpha\ll_{n} (t-s)^{1+\beta}
\]
for any $0\leq s, t\leq 1$. 
\end{proposition}
The proof of Proposition~\ref{propo_tightne2} is a consequence of the following series of lemmas.
\begin{lemma}\label{lemma_t1}
Let $n\geq 2$ be a fixed integer, $b_0$ be a fixed non-zero integer and $\alpha>0$. If
\[
0\leq t-s\leq\frac{1}{\varphi\left(p^n\right)-1}
\]then
\[
\frac{1}{\varphi\left(p^n\right)}\sum_{a\in\left(\Z/p^n\Z\right)^\times}\left\vert\mathsf{Kl}_{p^n}(t;(a,b_0))-\mathsf{Kl}_{p^n}(s;(a,b_0))\right\vert^\alpha\leq 2^\alpha(t-s)^{\alpha/2}.
\]
\end{lemma}
\begin{proof}[\proofname{} of Lemma~\ref{lemma_t1}]%
Note that
\begin{equation}\label{eq_csq1}
p^n\leq\frac{4}{t-s}.
\end{equation}
Let us assume that
\[
\frac{j-1}{\varphi\left(p^n\right)-1}<t\leq\frac{j}{\varphi\left(p^n\right)-1}
\]
where $1\leq j\leq\varphi\left(p^n\right)-1$. Two cases can occur. If
\[
\frac{j-1}{\varphi\left(p^n\right)-1}<s<t\leq\frac{j}{\varphi\left(p^n\right)-1}
\]
then by~\eqref{eq_21}
\begin{align*}
\left\vert\mathsf{Kl}_{p^n}(t;(a,b_0))-\mathsf{Kl}_{p^n}(s;(a,b_0))\right\vert & =\left\vert\alpha_j\left((a,b_0);p^n\right)\right\vert(t-s) \\
& \leq\frac{\varphi\left(p^n\right)-1}{p^{n/2}}(t-s) \\
& \leq 2\sqrt{t-s}
\end{align*}
by~\eqref{eq_csq1} and~\eqref{eq_tec_0}.
If
\[
\frac{j-2}{\varphi\left(p^n\right)-1}\leq s\leq\frac{j-1}{\varphi\left(p^n\right)-1}\leq t\leq\frac{j}{\varphi\left(p^n\right)-1}
\]
where $2\leq j\leq \varphi\left(p^n\right)-1$ then
\begin{multline*}
\left\vert\mathsf{Kl}_{p^n}(t;(a,b_0))-\mathsf{Kl}_{p^n}(s;(a,b_0))\right\vert\leq\left\vert\mathsf{Kl}_{p^n}(t;(a,b_0))-z_j\left((a,b_0);p^n\right)\right\vert \\
+\left\vert z_j\left((a,b_0);p^n\right)-\mathsf{Kl}_{p^n}(s;(a,b_0))\right\vert.
\end{multline*}
The first term is less than
\[
\left\vert\alpha_j\left((a,b_0);p^n\right)\right\vert\left(t-\frac{j-1}{\varphi\left(p^n\right)-1}\right)
\]
whereas the second term is less than
\[
\left\vert\alpha_{j-1}\left((a,b_0);p^n\right)\right\vert\left(\frac{j-1}{\varphi\left(p^n\right)-1}-s\right)
\]
which leads to
\[
\left\vert\mathsf{Kl}_{p^n}(t;(a,b_0))-\mathsf{Kl}_{p^n}(s;(a,b_0))\right\vert\leq 2\sqrt{t-s}
\]
by \eqref{eq_tec_0}.

This ensures the result.
\end{proof}

\begin{lemma}\label{lemma_t2}
Let $n\geq 2$ be a fixed integer, $b_0$ be a fixed non-zero integer and $\alpha\geq 1$. If
\[
t-s\geq\frac{1}{\varphi\left(p^n\right)-1}
\]
then
\begin{multline}\label{eq_lhs}
\frac{1}{\varphi\left(p^n\right)}\sum_{a\in\left(\Z/p^n\Z\right)^\times}\left\vert\mathsf{Kl}_{p^n}(t;(a,b_0))-\mathsf{Kl}_{p^n}(s;(a,b_0))\right\vert^\alpha\ll(t-s)^{\alpha/2} \\
+\frac{1}{\varphi\left(p^n\right)}\sum_{a\in\left(\Z/p^n\Z\right)^\times}\left\vert\widetilde{\mathsf{Kl}_{p^n}}(t;(a,b_0))-\widetilde{\mathsf{Kl}_{p^n}}(s;(a,b_0))\right\vert^\alpha.
\end{multline}
\end{lemma}
\begin{proof}[\proofname{} of Lemma~\ref{lemma_t2}]%
By~\eqref{eq_bemol},
\[
\left\vert\mathsf{Kl}_{p^n}(x;(a,b_0))-\widetilde{\mathsf{Kl}_{p^n}}(x;(a,b_0))\right\vert\leq\frac{6}{p^{n/2}}
\]
for any $0\leq x\leq 1$ and any $a$ in $\left(\Z/p^n\Z\right)^\times$. Thus, the left-hand side of~\eqref{eq_lhs} is bounded by
\[
\frac{1}{\varphi\left(p^n\right)}\sum_{a\in\left(\Z/p^n\Z\right)^\times}\left\vert\widetilde{\mathsf{Kl}_{p^n}}(t;(a,b_0))-\widetilde{\mathsf{Kl}_{p^n}}(s;(a,b_0))\right\vert^\alpha+O_\alpha\left(\frac{1}{p^{n}}\right)^{\alpha/2}
\]
which implies the result by the assumption on $t-s$.
\end{proof}
For $0\leqslant s, t\leqslant 1$, we define
\begin{equation}\label{eq_simp1}
j=\left\lceil{(\varphi\left(p^n\right)-1)s}\right\rceil\quad\text{ and }\quad k=\left\lceil{(\varphi\left(p^n\right)-1)t}\right\rceil
\end{equation}
such that
\begin{equation}\label{eq_simp2}
\frac{j-1}{\varphi\left(p^n\right)-1}<s\leq\frac{j}{\varphi\left(p^n\right)-1}\quad\text{ and }\quad\frac{k-1}{\varphi\left(p^n\right)-1}<t\leq\frac{k}{\varphi\left(p^n\right)-1}.
\end{equation}
These notations will be used in the proofs of Lemma~\ref{lemma_t4}, Lemma~\ref{lemma_t5} and Lemma~\ref{lemma_t3}.
\begin{lemma}\label{lemma_t4}
Let $n\geq 2$ be a fixed integer, $b_0$ be a fixed non-zero integer and $\alpha\geq 1$. If
\[
\frac{1}{\varphi\left(p^n\right)-1}\leq t-s\leq\frac{1}{p^{n/2+\delta}}
\]
where $0<\delta<n/2$ then
\begin{align*}
\frac{1}{\varphi\left(p^n\right)}\sum_{a\in\left(\Z/p^n\Z\right)^\times}\left\vert\mathsf{Kl}_{p^n}(t;(a,b_0))-\mathsf{Kl}_{p^n}(s;(a,b_0))\right\vert^\alpha&\\
\ll(t-s)^{\alpha/2}&+(t-s)^{\alpha\delta/n}.
\end{align*}
\end{lemma}

\begin{proof}[\proofname{} of Lemma~\ref{lemma_t4}]%
Recall \eqref{eq_simp1} and \eqref{eq_simp2}. By~\eqref{eq_Kltilde}, one trivially gets
\[
\left\vert\widetilde{\mathsf{Kl}_{p^n}}(t;(a,b_0))-\widetilde{\mathsf{Kl}_{p^n}}(s;(a,b_0))\right\vert\leq\frac{\abs{\cI_{s,t}}}{p^{n/2}}
\]
where $\cI_{s,t}$ is the non-empty interval in $\Z$ given by
\begin{equation}\label{eq_Ist}
\left(x_j(s)=\varphi\left(p^n\right)s+j-1\right.,\left.x_k(t)=\varphi\left(p^n\right)t+k-1\right].
\end{equation} 
Its length satisfies
\begin{equation}\label{eq_Length Ist}
\begin{split}
\abs{\cI_{s,t}} & =\left\lfloor{x_k(t)}\right\rfloor-\left\lceil{x_j(s)}\right\rceil \\
& \leq\varphi\left(p^n\right)(t-s)+\left\lceil{(\varphi\left(p^n\right)-1)t}\right\rceil-\left\lceil{(\varphi\left(p^n\right)-1)s}\right\rceil \\
& \leq 4(\varphi\left(p^n\right)-1)(t-s)+1 \\
& \leq 8(\varphi\left(p^n\right)-1)(t-s)
\end{split}
\end{equation}
since $(\varphi\left(p^n\right)-1)(t-s)\geq 1$. Thus,
\[
\left\vert\widetilde{\mathsf{Kl}_{p^n}}(t;(a,b_0))-\widetilde{\mathsf{Kl}_{p^n}}(s;(a,b_0))\right\vert\leq 8p^{-\delta}.
\]
This implies the desired result by Lemma \ref{lemma_t2}.
\end{proof}


We recall the definitions of the constants $\gamma_1$ and $\gamma_2$ from Proposition~\ref{propo_koro}. 

\begin{lemma}\label{lemma_t5}
Let $n\geq 31$ be a fixed integer, $b_0$ be a fixed non-zero integer and $\alpha\geq 1$. If
\[
\frac{1}{p^{n/2+\delta}}\leq t-s\leq\frac{1}{p^{n/2-\delta}}
\]
where $0<\delta\leq\min{\left(\gamma_2n/16,n/2-15\right)}$ then
\begin{align*}
\frac{1}{\varphi\left(p^n\right)}\sum_{a\in\left(\Z/p^n\Z\right)^\times}\left\vert\mathsf{Kl}_{p^n}(t;(a,b_0))-\mathsf{Kl}_{p^n}(s;(a,b_0))\right\vert^\alpha&\\
\ll(t-s)^{\alpha/2}& +(t-s)^{\alpha\delta/(n/2+\delta)}.
\end{align*}
\end{lemma}

\begin{proof}[\proofname{} of Lemma~\ref{lemma_t5}]%
Recall \eqref{eq_simp1} and \eqref{eq_simp2}. Once again,
\[
\widetilde{\mathsf{Kl}_{p^n}}(t;(a,b_0))-\widetilde{\mathsf{Kl}_{p^n}}(s;(a,b_0))=\frac{1}{p^{n/2}}\sum_{x\in \cI_{s,t}}e_{p^n}\left(ax+b_0\overline{x}\right)
\]
by~\eqref{eq_Kltilde} for any $a$ in $\left(\Z/p^n\Z\right)^\times$ and where 
$\cI_{s,t}$ is given by~\eqref{eq_Ist}.
According to~\eqref{eq_Length Ist} 
for its length $\abs{\cI_{s,t}}$  we have
\[ 
p^{n/2+\delta} \gg \abs{\cI_{s,t}} \gg p^{n/2-\delta}.
\]

By Corollary~\ref{coro_koro},
\[
\left\vert\widetilde{\mathsf{Kl}_{p^n}}(t;(a,b_0))-\widetilde{\mathsf{Kl}_{p^n}}(s;(a,b_0))\right\vert\ll\left(\frac{1}{p^n}\right)^{\delta/n}\ll(t-s)^{\frac{\delta}{n/2+\delta}}
\]
for any $a$ in $\left(\Z/p^n\Z\right)^\times$, which implies the result.
\end{proof}
\begin{lemma}\label{lemma_t3}
Let $n\geq 2$ be a fixed integer, $b_0$ be a fixed non-zero integer and $\alpha$ be a non-zero even integer. If
\[
\frac{1}{p^{n/2-\delta}}\leq t-s\leq 1
\]
where $0<\delta<n/2$ then
\begin{align*}
\frac{1}{\varphi\left(p^n\right)}\sum_{a\in\left(\Z/p^n\Z\right)^\times}&\left\vert\mathsf{Kl}_{p^n}(t;(a,b_0))-\mathsf{Kl}_{p^n}(s;(a,b_0))\right\vert^\alpha\\
&\qquad \qquad  \ll(t-s)^{\alpha/2}+(t-s)^{1+\delta/(n/2-\delta)-\varepsilon}
\end{align*}
for any $\varepsilon>0$.
\end{lemma}
\begin{proof}[\proofname{} of Lemma~\ref{lemma_t3}]%
Let us define for any $0\leq x\leq 1$ the random variables
\[
\mathsf{Kl}_{p^n}\left(x;\frac{p^n-1}{2};\ast\right)\coloneqq\frac{1}{p^{n/2}}\sum_{\abs{h}\leq (p^n-1)/2}\alpha_{p^n}(h;x)U_h(\ast).
\] 
Recall that each random variable $U_h$, $h\in\Z$, is $4$-subgaussian since it is centered and bounded by $2$ (see~\cite[Proposition~B.6.2]{Ko}). It turns out that for any real number $u$,
\begin{align*}
\mathbb{E}\left(e^{u\left(\mathsf{Kl}_{p^n}\left(t;\frac{p^n-1}{2};\ast\right)-\mathsf{Kl}_{p^n}\left(s;\frac{p^n-1}{2};\ast\right)\right)}\right) & =\prod_{\abs{h}\leq (p^n-1)/2}\mathbb{E}\left(e^{u\frac{\alpha_{p^n}(h;t)-\alpha_{p^n}(h;s)}{p^{n/2}}U_h(\ast)}\right) \\
& \leq\prod_{\abs{h}\leq (p^n-1)/2}e^{\frac{4\left\vert\alpha_{p^n}(h;t)-\alpha_{p^n}(h;s)\right\vert^2}{p^n}u^2/2}
\end{align*}
by the independence of the random variables $U_h$, $h\in\Z$. Thus, by definition, the random variable $\mathsf{Kl}_{p^n}\left(t;\frac{p^n-1}{2};\ast\right)-\mathsf{Kl}_{p^n}\left(s;\frac{p^n-1}{2};\ast\right)$ is $\sigma_{p^n}$-sub\-gaussian,
 where
\[
\sigma_{p^n}^2=\frac{4}{p^n}\sum_{\abs{h}\leq (p^n-1)/2}\left\vert\alpha_{p^n}(h;t)-\alpha_{p^n}(h;s)\right\vert^2.
\] 
Consequently,
\[
\mathbb{E}\left(\left\vert\mathsf{Kl}_{p^n}\left(t;\frac{p^n-1}{2};\ast\right)-\mathsf{Kl}_{p^n}\left(s;\frac{p^n-1}{2};\ast\right)\right\vert^\alpha\right)\leq c_\alpha\sigma_{p^n}^\alpha
\]
for some positive constant $c_\alpha$ by \cite[Proposition B.6.3]{Ko}.
\par
Recall \eqref{eq_simp1} and \eqref{eq_simp2}. By \eqref{eq_Fourier} and the discrete Plancherel formula,
\[
\sigma_{p^n}^2=\frac{4}{p^n}\abs{\cI_{s,t}}
\]
where $\cI_{s,t}$
is the non-empty interval in $\left(\Z/p^n\Z\right)^\times$ given by~\eqref{eq_Ist}
whose  length satisfies~\eqref{eq_Length Ist}.
Thus, 
\begin{equation}\label{eq_steppp}
\mathbb{E}\left(\left\vert\mathsf{Kl}_{p^n}\left(t;\frac{p^n-1}{2}\ast\right)-\mathsf{Kl}_{p^n}\left(s;\frac{p^n-1}{2};\ast\right)\right\vert^\alpha\right)\leq 32^{\alpha/2}c_\alpha(t-s)^{\alpha/2}.
\end{equation}

The same method of proof than the one used in~\cite[Proposition~4.1]{MR3854900} entails that
\begin{align*}
\frac{1}{\varphi\left(p^n\right)}& \sum_{a\in\left(\Z/p^n\Z\right)^\times}\left\vert\widetilde{\mathsf{Kl}_{p^n}}(t;(a,b_0))-\widetilde{\mathsf{Kl}_{p^n}}(s;(a,b_0))\right\vert^\alpha \\
& =\mathbb{E}\left(\left\vert\mathsf{Kl}_{p^n}\left(t;\frac{p^n-1}{2};\ast\right)-\mathsf{Kl}_{p^n}\left(s;\frac{p^n-1}{2};\ast\right)\right\vert^\alpha\right)+O\left(\frac{\log^\alpha{\left(p^n\right)}}{p^{n/2}}\right)\\
& \ll(t-s)^{\alpha/2}+(t-s)^{1+\delta/(n/2-\delta)-\varepsilon}
\end{align*}
by~\eqref{eq_steppp} for any $\varepsilon>0$. See \cite[Page 322]{MR3854900} for more details.

One can apply Lemma~\ref{lemma_t2} to conclude the proof.
\end{proof}
\begin{proof}[\proofname{} of Proposition~\ref{propo_tightne2}]%
Obviously, one can assume that
\[
0\leq s<t\leq 1.
\]

Let $\delta$ be any real number satisfying
\[
0<\delta\leq\min{\left(\gamma_2n/16,n/2-15\right)}
\]
and $\alpha$ be any even integer satisfying
\[
\alpha>\max{\left(\frac{n}{\delta},\frac{n/2+\delta}{\delta}\right)}.
\] 
Let us define
\[
\beta=\min{\left(\frac{\alpha}{2},\frac{\alpha\delta}{n},2,\frac{\delta\alpha}{n/2+\delta},1+\frac{\delta}{n/2-\delta}\right)}-1>0
\]
and let us prove that
\[
\frac{1}{\varphi\left(p^n\right)}\sum_{a\in\left(\Z/p^n\Z\right)^\times}\left\vert\mathsf{Kl}_{p^n}(t;(a,b_0))-\mathsf{Kl}_{p^n}(s;(a,b_0))\right\vert^\alpha\ll_{\alpha, \delta, n, \varepsilon}\abs{t-s}^{1+\beta-\varepsilon}
\]
for any $\varepsilon>0$.


By Lemmas~\ref{lemma_t1}, \ref{lemma_t4} and~\ref{lemma_t5},
\[
\frac{1}{\varphi\left(p^n\right)}\sum_{a\in\left(\Z/p^n\Z\right)^\times}\left\vert\mathsf{Kl}_{p^n}(t;(a,b_0))-\mathsf{Kl}_{p^n}(s;(a,b_0))\right\vert^\alpha\ll\abs{t-s}^{1+\beta}
\]
provided that
\[
0\leq t-s\leq\frac{1}{p^{n/2-\delta}}.
\]

Let us assume from now on that
\[
\frac{1}{p^{n/2-\delta}}\leq t-s\leq 1.
\]
By Lemma~\ref{lemma_t3},
\begin{multline*}
\frac{1}{\varphi\left(p^n\right)}\sum_{a\in\left(\Z/p^n\Z\right)^\times}\left\vert\mathsf{Kl}_{p^n}(t;(a,b_0))-\mathsf{Kl}_{p^n}(s;(a,b_0))\right\vert^\alpha \\
\ll_\alpha\left((t-s)^{\alpha/2}+(t-s)^{1+\delta/(n/2-\delta)-\epsilon} \right)\ll_\epsilon(t-s)^{1+\beta-\epsilon}.
\end{multline*}
\end{proof}
\section{Proof of Theorems~\ref{theo_A} and~\ref{theo_B}}\label{sec_proofs}%
Theorem~\ref{theo_B} follows from Proposition~\ref{propo_tightne2} by Kolmogorov's criterion for tightness (see~\cite[Proposition~A.1]{MR3854900}).
\begin{remark}
The proof of tightness in \cite[Lemma 3.2]{KoSa} has a small mistake at the end, since the authors use the uniform bound
\begin{equation*}
\widetilde{L}_p(t;\omega)\ll\log{(p)}
\end{equation*}
to uniformize the exponent $\alpha$ across different estimates in order to apply Kolmogorov's criterion for tightness (see~\cite[Proposition~A.1]{MR3854900}). This introduces an overlooked dependency on $p$ when $t-s$ is close to $1$. One can correct this problem either arithmetically (as done in the present paper) by showing that one can take $\alpha$ to be a sufficiently large even integer, or (as suggested by E.~Kowalski) by proving a generalization of Kolmogorov's tightness criterion that involves different exponents $\alpha$ in different ranges.
\end{remark}
Theorem~\ref{theo_A} follows from Theorem~\ref{theo_B} and~\cite[Theorem~A]{MR3854900} by Prokho\-rov's criterion for the convergence in law (see~\cite[Theorem~A.3]{MR3854900}).
\section*{Acknowlegment}
The authors would like to thank E.~Kowalski for sharing with us his enlightening lecture notes~\cite{Ko} and for helping them with Lemma~\ref{lemma_t3} especially when $s-t$ is very close to $1$.

The first and second authors are financed by the ANR Project Flair ANR-17-CE40-0012.

The third author is partially supported by the  Australian Research Council Grant DP170100786.

Last but not least, the authors would like to thank the anonymous referee for a careful reading of the manuscript, for suggesting them Corollary \ref{propo_koro_2} and for pointing out a naive mistake in a previous version of Corollary \ref{coro_koro}. These comments made the manuscript more readable.
\bibliographystyle{alpha}

\end{document}